\documentclass[a4paper,          
               12pt,             
               notitlepage]{amsart}

\usepackage[utf8]{inputenc}
\usepackage{amsmath}
\usepackage{amsthm}
\usepackage{amsfonts}
\usepackage{amssymb}

\usepackage[pdftex]{graphicx}
\usepackage[pdftex]{color}

\usepackage{url}
\usepackage{hyperref}
\definecolor{darkblue}{RGB}{0,0,160}
\hypersetup{
colorlinks,%
citecolor=black,%
filecolor=black,%
linkcolor=darkblue,%
urlcolor=darkblue
}

\usepackage{verbatim}

\author{Thomas Kahle}
\address{Fakultät für Mathematik, Otto-von-Guericke Universität\\ Magdeburg, Germany}
\email{www.thomas-kahle.de}
\author{Johannes~Rauh}
\address{MPI MIS \\ Leipzig, Germany}
\email{jrauh@mis.mpg.de}

\title{Toric fiber products versus Segre products}
\date{April 28, 2014}

\newcommand{\Kb}{\mathbb{K}}
\newcommand{\Nb}{\mathbb{N}}
\newcommand{\Pb}{\mathbb{P}}

\newcommand{\Zb}{\mathbb{Z}}

\newcommand{\Acal}{\mathcal{A}}
\newcommand{\Bcal}{\mathcal{B}}
\newcommand{\Ccal}{\mathcal{C}}
\newcommand{\Fcal}{\mathcal{F}}
\newcommand{\Gcal}{\mathcal{G}}

\newcommand{\Lcal}{\mathcal{L}}
\newcommand{\Vcal}{\mathcal{V}}

\def\ol#1{{\overline {#1}}}

\newcommand{\oxab}{\ol x^\alpha_\beta}
\newcommand{\oyag}{\ol y^\alpha_\gamma}

\newcommand{\bal}{{\boldsymbol{\alpha}}}
\newcommand{\bbet}{{\boldsymbol{\beta}}}
\newcommand{\bgam}{{\boldsymbol{\gamma}}}
\newcommand{\biot}{{\boldsymbol{\iota}}}

\newcommand{\SP}[1][\Nb\Acal]{\mathbin{\#_{#1}}}
\newcommand{\TFP}[1][\Acal]{\times_{#1}}

\DeclareMathOperator{\codim}{codim}

\DeclareMathOperator{\Spec}{Spec}

\DeclareMathOperator{\gp}{gp}

\newtheorem{thm}{Theorem} 
\newtheorem{lemma}[thm]{Lemma}

\newtheorem{cor}[thm]{Corollary}
\newtheorem{prop}[thm]{Proposition}

\theoremstyle{definition}
\newtheorem{ex}[thm]{Example}
\newtheorem{rem}[thm]{Remark}
\newtheorem{defi}[thm]{Definition}

\newtheorem*{conv}{Conventions}

\newlength{\wwolengthA}
\newlength{\wwolengthB}
\newcommand{\wwo}[2]{\relax\ifmmode\settowidth{\wwolengthA}{$#1$}\settowidth{\wwolengthB}{$#2$}\else\settowidth{\wwolengthA}{#1}\settowidth{\wwolengthB}{#2}\fi%
  \addtolength{\wwolengthB}{-\wwolengthA}%
  \hspace{0.5\wwolengthB}#1\hspace{0.5\wwolengthB}}
\newcommand{\gama}{\wwo{\gamma}{\beta}}

\begin{document}

\makeatletter
  \@namedef{subjclassname@2010}{\textup{2010} Mathematics Subject Classification}
\makeatother
\subjclass[2010]{Primary: 13C05; Secondary: 05E40, 13A02, 14M25, 20M25}


\begin{abstract}
  The toric fiber product is an operation that combines two ideals that are homogeneous with respect to a grading by an
  affine monoid.  The Segre product is a related construction that combines two multigraded rings.  The quotient ring by
  a toric fiber product of two ideals is a subring of the Segre product, but in general this inclusion is strict.  We
  contrast the two constructions and show that any Segre product can be presented as a toric fiber product without
  changing the involved quotient rings.  This allows to apply previous results about toric fiber products to the study
  of Segre products.

  We give criteria for the Segre product of two affine toric varieties to be dense in their toric fiber product, and for
  the map from the Segre product to the toric fiber product to be finite.  We give an example that shows that the
  quotient ring of a toric fiber product of normal ideals need not be normal.  In rings with Veronese type gradings, we
  find examples of toric fiber products that are always Segre products, and we show that iterated toric fiber products
  of Veronese ideals over Veronese rings are normal.
\end{abstract}

\maketitle

\section{Introduction}
\label{sec:intro}

For each positive integer $d$ let $[d]:= \{1,\dots,d\}$.  Let $\Acal$ be a vector configuration of $d$ integer vectors
$a_{1},\dots,a_{d}\in\Zb^{h}$, for some~$h>0$.  Denote by~$\Nb\Acal$ the \emph{affine monoid} generated by~$\Acal$.  The
\emph{codimension} of $\Acal$ is $\codim(\Acal)=h-\dim(\Nb\Acal)$.
It is often convenient to think of $\Acal$ as an $(h\times d)$-matrix, with columns indexed by~$\alpha\in[d]$.
Therefore, the \emph{lattice} $\ker_{\Zb}\Acal:=\{c\in\Zb^{d}:\sum_{\alpha=1}^{d}c_{\alpha}a_{\alpha}=0\}$ is called the
\emph{integer kernel} of~$\Acal$.  With this notation $\codim(\Acal)=\dim(\ker_{\Zb}(\Acal))$.

The following condition on the matrix $\Acal$ has many nice implications and often appears in applications; nevertheless
we do not impose it throughout:
\begin{equation}
  \label{eq:condA}\tag{$\ast$}
  \text{The vectors }\Acal\text{ lie in an affine hyperplane that does not pass through the origin.}
\end{equation}
If $\Acal$ is interpreted as a matrix, then \eqref{eq:condA} is satisfied if and only if the row space of $\Acal$
contains the vector $(1,1,\dots,1)$.

\begin{conv}
  When working with many indices we use simple superscripts to denote some of them.  Throughout the paper $x^{\alpha}$
  stands for a variable with superscript index~$\alpha$.  The square of a variable with $\alpha=1$, for instance, is
  denoted by~$(x^{1})^{2}$.  In this paper all tensor products $\otimes$ are tensor products over the ground field~$\Kb$.
\end{conv}

For each $\alpha\in[d]$ let $d_{\alpha},d'_{\alpha}\ge 0$ be non-negative integers.  Consider the polynomial rings
\begin{align*}
  \Kb[x] &= \Kb[x^{\alpha}_{\beta} : \alpha\in[d],\beta\in[d_{\alpha}]] \\
  \Kb[y] &= \Kb[y^{\alpha}_{\gamma} : \alpha\in[d],\gamma\in[d'_{\alpha}]]\,,
\end{align*}
where $\Kb$ is a field.
Define an $\Nb\Acal$-grading on $\Kb[x]$ and $\Kb[y]$ via
$\deg(x^{\alpha}_{\beta})=\deg(y^{\alpha}_{\gamma})=a_{\alpha}$.  If~\eqref{eq:condA} is satisfied, then every
$\Nb\Acal$-homogeneous ideal is also homogeneous with respect to the total $\Nb$-grading under which every variable
$x^\alpha_\beta,y^\alpha_\gamma$ has degree one.

Consider an additional ring $\Kb[z] = \Kb[z^{\alpha}_{\beta,\gamma} :
\alpha\in[d],\beta\in[d_{\alpha}],\gamma\in[d'_{\alpha}]]$ with grading $\deg(z^{\alpha}_{\beta,\gamma})=a_{\alpha}$,
and let $\phi_{S}:\Kb[z]\to\Kb[x]\otimes\Kb[y]$ be the ring homomorphism defined by
$\phi_{S}(z^{\alpha}_{\beta,\gamma})=x^{\alpha}_{\beta}\otimes y^{\alpha}_{\gamma}$.  If $d=1$, then $\phi_{S}$
describes the Segre embedding
\begin{equation*}
  \Pb^{d_{1}-1}\Kb\times\Pb^{d'_{1}-1}\Kb \hookrightarrow \Pb^{d_{1}d'_{1}-1}\Kb.
\end{equation*}
In general, $\phi_{S}$ describes a product of $d$ Segre embeddings.

Let $I\subseteq\Kb[x]$ and $J\subseteq\Kb[y]$ be two $\Nb\Acal$-homogeneous ideals, and denote by
\begin{align*}
  \pi_{I} &:\Kb[x]\to R:=\Kb[x]/I \\
  \pi_{J} &:\Kb[\wwo{y}{x}]\to \wwo{S}{R}:=\Kb[\wwo{y}{x}]/J
\end{align*}
the canonical projections.  Their tensor product as $\Kb$-algebra homomorphisms is
$\pi_{I}\otimes\pi_{J}:\Kb[x]\otimes\Kb[y]\to R\otimes S$.  The images of $x^{\alpha}_{\beta}$ and $y^{\alpha}_{\gamma}$
are $\pi_{I}(x^{\alpha}_{\beta})=\oxab$ and~$\pi_{J}(y^{\alpha}_{\gamma})=\oyag$ and the image of
$x^{\alpha}_{\beta}\otimes y^{\alpha}_{\gamma}\in\Kb[x]\otimes\Kb[y]$ under $\pi_{I}\otimes\pi_{J}$ is
$\oxab\otimes\oyag$.
\begin{defi}\label{d:fiber}
  The \emph{toric fiber product} of $I$ and $J$ is the ideal
  $I\times_{\Acal}J:=\ker((\pi_{I}\otimes\pi_{J})\circ\phi_{S})$.
\end{defi}
\begin{defi}\label{d:segre}
  The \emph{(multigraded) Segre product} of $R$ and $S$ is $R\SP S := \sum_{a\in\Nb\Acal}R_{a}\otimes S_{a}$.
\end{defi}
The toric fiber product was first defined in~\cite{sullivant07:_toric} as a unifying generalization of different
constructions of algebraic statistics.  Further progress was made in~\cite{TFP-II}.  The multigraded Segre product is a
generalization of the construction that is used to present the product of projective varieties as a projective variety.
Definition~\ref{d:segre} appered in~\cite{TFP-II} but the multigraded generalization is natural, and (a variant) appears
already in the work of Chow~\cite{chow1964unmixedness}.  Recent interest in this construction comes from the study of
generators of \emph{phylogenetic semigroups}~\cite{buczynska12}.  In the following we omit the attribute multigraded.

The Segre product is a product of rings.  Whenever $R\cong R'$ and $S\cong S'$, then $R\SP S \cong R'\SP S'$.  In
constrast, the toric fiber product is a product of ideals and depends on the chosen presentations.  If $R = \Kb[x]/I$
and $S = \Kb[y]/J$, then $\Kb[z]/(I\TFP J)$ depends on choices made in the presentation, as the next
Example~\ref{ex:dep-on-pres} shows.  One aim of this paper is to make some of the more subtle consequences of this fact
explicit.

\begin{ex}\label{ex:dep-on-pres}
  Let $\Acal=(1,\dots,1)$ and $d_{\alpha}=1=d'_{\alpha}$ for all~$\alpha\in[d]$.  In this case,
  $\Kb[x]\cong\Kb[y]\cong\Kb[z]$.  If $I\subset\Kb[x]$ and $J\subset\Kb[y]$ are monomial ideals, then their toric fiber
  product is just their sum after variable substitution.
  With $d=2$ consider the ideals
  \begin{equation*}
	I_{1}= \langle x^{1}\rangle, \qquad  I_{2}= \langle x^{2}\rangle, \qquad  J= \langle y^{1}\rangle.
  \end{equation*}
  Then the three rings
  \begin{equation*}
	R_{1} :=\Kb[x^{1},x^{2}]/I_{1}, \qquad 
	R_{2} :=\Kb[x^{1},x^{2}]/I_{2}, \qquad 
	S:=\Kb[y^{1},y^{2}]/J
  \end{equation*}
  are all isomorphic (as $\Nb\Acal$-graded rings) to a polynomial ring with a single variable of degree one.  Now
  $I_{1}\times_{\Acal}J\cong\langle z^{1}\rangle$, and so $\Kb[z]/(I_{1}\times_{\Acal}J)\cong R_{1}$.  On the other hand,
  $I_{2}\times_{\Acal}J\cong\langle z^{1},z^{2}\rangle$, and so $\Kb[z]/(I_{2}\times_{\Acal}J)\cong \Kb$.  The Segre
  product is $R_{1}\SP S\cong R_{2}\SP S\cong \Kb[z]/(I_{1}\times_{\Acal}J)$.
\end{ex}

One could define a ``toric fiber product'' of $\Kb$-algebras $R, S$ with homogeneous generating sets
$\{x^{\alpha}_{\beta}\}\subset R$ and $\{y^{\alpha}_{\gamma}\}\subset S$ as the subring of $R\SP S$ generated by the
images $\oxab\otimes\oyag$ of $x^{\alpha}_{\beta}\otimes y^{\alpha}_{\gamma}$.  With a suitable choice of the generating
sets, it is possible to present the Segre product as the quotient ring of such a product
(Proposition~\ref{prop:SP-as-TFP}).  As it turns out, every Segre product $R\SP S$ of $\Nb\Acal$-graded affine rings
contains a family of toric fiber products corresponding to presentations of $R$ and $S$ as quotients of polynomial
rings.  The Segre product itself is also of this form.

Example~\ref{ex:dep-on-pres} also shows that the toric fiber product depends implicitly on the choice of the vector
configuration $\Acal\subset\Nb\Acal$; most notably in the case where multiple identical vectors are allowed.
Lemma~\ref{lem:canonical-TFP} shows that the situation is slightly better if~\eqref{eq:condA} is satisfied.

The Segre product of two affine rings is again an affine ring (Section~\ref{sec:generators-SP}).  A particular finite
generating set is given in Proposition~\ref{prop:Segre-generators}.  Any such presentation of $R\SP S$ as a quotient of
some polynomial ring defines an ideal.  In Section~\ref{sec:relations} we compare this ideal to the toric fiber product
of ideals presenting $R$ and~$S$, and we show that the corresponding quotient rings form a partially ordered set with
maximal element~$R\SP S$.  In some cases (Sections~\ref{sec:toric-ideals} and~\ref{exs:Veronese-products}) the relations
presenting $R\SP S$ can be found explicitly.

If $\Acal$ is fixed, then one can take iterated toric fiber products or Segre products.  In
Section~\ref{sec:neutral-elements} we show that the toric ideal of $\Acal$ equals the neutral element with respect to
both the toric fiber product and the Segre product.  Consequently the set of $\Nb\Acal$-graded rings forms a monoid
under each of the operations.  Section~\ref{sec:toric-ideals} is dedicated to the case of toric ideals and affine monoid
rings.  Section~\ref{sec:examples} collects examples that illustrate our results.

\section{Generators of the Segre product}
\label{sec:generators-SP}

The Segre product of affine $\Kb$-algebras is again an affine $\Kb$-algebra.  This follows since the Segre product can
be described as the coordinate ring of a quotient of $\Spec R\times\Spec S$ by an algebraic torus action,
see~\cite{TFP-II}.  In the following we explicitly describe a finite set of generators.

The map $\pi_{I}\otimes\pi_{J}$ is homogeneous and surjective, and so $R_{a}\otimes S_{a}$ is a homomorphic image of
$\Kb[x]_{a}\otimes\Kb[y]_{a}$.  For arbitrary multiindices $\bal=(\alpha_{1},\dots,\alpha_{n})$,
$\bal'=(\alpha'_{1},\dots,\alpha'_{m})$, $\bbet=(\beta_{1},\dots,\beta_{n})$ and $\bgam=(\gamma_{1},\dots,\gamma_{m})$
with $\beta_{i}\in[d_{\alpha_{i}}]$ and $\gamma_{j}\in[d'_{\alpha'_{j}}]$ let
\begin{equation*}
  M^{\bal,\bal'}_{\bbet,\bgam}
  = \Big(\prod_{i=1}^{n}x^{\alpha_{i}}_{\beta_{i}}\Big)\otimes \Big(\prod_{j=1}^{m}y^{\alpha'_{j}}_{\gamma_{j}}\Big),
\end{equation*}
Any element of $\Kb[x]_{a}\otimes\Kb[y]_{a}$ is a sum of monomials of the form $M^{\bal,\bal'}_{\bbet,\bgam}$, where the
multiindices $\bal,\bal'$ are such that the vector $\tilde c = \sum_{i=1}^{n}e_{\alpha_{i}} -
\sum_{j=1}^{m}e_{\alpha'_{j}}$ belongs to the kernel of~$\Acal$.  If $n=m=1$ and $\alpha_{1}=\alpha'_{1}=\alpha$, then
$M^{\alpha,\alpha}_{\beta,\gamma}=x^{\alpha}_{\beta}\otimes y^{\alpha}_{\gamma}$ is a \emph{\hypertarget{simpMon}{simple
    monomial}}.  The image $\ol M^{\alpha,\alpha}_{\beta,\gamma}$ of $M^{\alpha,\alpha}_{\beta,\gamma}$ under
$\pi_{I}\otimes\pi_{J}$ is a \emph{\hypertarget{simpGen}{simple generator}}.  Observe that if the ring $\Kb[z]/(I\TFP
J)$ is identified with its image as a subring of $R\otimes S$, then it is generated by the simple generators.  The
following definition is needed to find generators of $R\SP S$:
\begin{defi}
  The \emph{Graver basis} $\Gcal$ of $\ker_{\Zb}\Acal$ is the unique minimal subset of $\ker_{\Zb}\Acal$ such that any
  $c\in\ker_{\Zb}\Acal$ has a finite sign-consistent representation $c=\sum_{i}g_{i}$ with $g_{i}\in\Gcal$.  Here,
  sign-consistency means that $(g_{i})_{\alpha}c_{\alpha}\ge 0$ for all $i$ and~$\alpha\in[d]$.
\end{defi}
Graver bases exist 
and are unique, because they consist exactly
of the \emph{primitive} vectors in $\ker_{\Zb}\Acal$
(see~\cite[Chapter~7]{sturmfels96:_gr_obner_bases_and_convex_polyt}).  Graver bases are of fundamental importance in
integer optimization~\cite[Part~II]{deLoeraHemmecke13}.
Let $\Gcal$ be the Graver basis of~$\ker_{\Zb}\Acal$, and let~$g\in\Gcal$.  Write $g = \sum_{i=1}^{n}e_{\alpha_{i}} -
\sum_{j=1}^{m}e_{\alpha'_{j}}$ as an integral combination of unit vectors, where
$n=\sum_{\alpha:g_{\alpha}>0}|g_{\alpha}|$ and~$m=\sum_{\alpha:g_{\alpha}<0}|g_{\alpha}|$, and let
$\bal=(\alpha_{1},\dots,\alpha_{n})$ and $\bal'=(\alpha'_{1},\dots,\alpha'_{m})$.  This decomposition is unique under
the assumption that $\alpha_{i}\le \alpha_{i+1}$ and $\alpha'_{j}\le\alpha'_{j+1}$.  For each $i=1,\dots,n$ and
$j=1,\dots,m$ let $\beta_{i}\in[d_{\alpha_{i}}]$ and~$\gamma_{j}\in[d'_{\alpha_{j}}]$.  Then we call
$\bbet=(\beta_{1},\dots,\beta_{n})$, $\bgam=(\gamma_{1},\dots,\gamma_{m})$ a \emph{pair of \hypertarget{GIS}{Graver
    index sequences}} of~$g$.  For an arbitrary pair of Graver index sequences $\bbet$ and $\bgam$ we call
$M^{g}_{\bbet,\bgam}:=M^{\bal,\bal'}_{\bbet,\bgam}$ a \emph{Graver monomial}.  The image $\ol
M^{g}_{\bbet,\bgam}=\pi_{I}\otimes\pi_{J}(M^{g}_{\bbet,\bgam})$ is a \emph{Graver generator}.
\begin{prop}
  \label{prop:Segre-generators}
  The Segre product $R\SP S$ is generated, as a $\Kb$-algebra, by the simple generators and the Graver
  generators.
\end{prop}
\begin{proof}
  Since $\Gcal\subset\ker_{\Zb}\Acal$, every Graver generator belongs to~$R\SP S$.  It suffices to show that any
  monomial of the form $M^{\bal,\bal'}_{\bbet,\bgam}$, where $\tilde c = \sum_{i=1}^{n}e_{\alpha_{i}} -
  \sum_{j=1}^{m}e_{\alpha'_{j}}$ belongs to the kernel of~$\Acal$, is a product of simple monomials and Graver
  monomials.  If the intersection $\bal\cap\bal'$ is not empty, say
  $\bal\cap\bal'=(\alpha_{1},\dots,\alpha_{r})=(\alpha'_{1},\dots,\alpha'_{r})$ (as an intersection of multisets), then
  we can factor out a product of simple monomials and write
  \begin{equation*}
    M^{\bal,\bal'}_{\bbet,\bgam}
    =\prod_{i=1}^{r}(x^{\alpha_{i}}_{\beta_{i}}\otimes y^{\alpha_{i}}_{\gamma_{i}})\cdot
    M^{(\alpha_{r+1},\dots,\alpha_{n}),(\alpha'_{r+1},\dots,\alpha'_{m})}_{(\beta_{r+1},\dots,\beta_{n}),(\gamma_{r+1},\dots,\gamma_{m})}.
  \end{equation*}
  Hence, we may assume $\{\alpha_{i}\}\cap\{\alpha'_{j}\}=\emptyset$.  It suffices to show that all such monomials can
  be expressed as products of Graver monomials.  Since $\tilde c$ lies in the kernel of~$\Acal$, it can be written as a
  sign-consistent sum $\tilde c = \sum_{s=1}^{t}g_{s}$ of elements~$g_{s}\in\Gcal$.  Reorder the indices $\alpha_{i}$
  and $\alpha'_{j}$ such that $g_{s}=\sum_{i=n_{s-1}+1}^{n_{s}}e_{\alpha_{i}} -
  \sum_{j=m_{s-1}+1}^{m_{s}}e_{\alpha'_{j}}$.  Then
  \begin{equation*}
    M^{\bal,\bal'}_{\bbet,\bgam}
    = \prod_{s=1}^{t} M^{(\alpha_{n_{s-1}+1},\dots,\alpha_{n_{s}}),(\alpha'_{m_{s-1}+1},\dots,\alpha'_{m_{s}})}_{(\beta_{n_{s-1}+1},\dots,\beta_{n_{s}}),(\gamma_{m_{s-1}+1},\dots,\gamma_{m_{s}})},
  \end{equation*}
  where the right hand side is a product of Graver monomials.
\end{proof}
If $I=0$ and $J=0$, then the generating set of $R\SP S$ presented in Proposition~\ref{prop:Segre-generators} is minimal.
Otherwise, the relations in $I$ and $J$ may allow to express some generators in terms of the other generators.  

\begin{ex}
  \label{ex:Segre-generators}
  Let $d=3$, $\Acal=(1,1,1)$, and consider the rings $R=\Kb[x^{1},x^{2},x^{3}]/\langle x^{1}\rangle$ and
  $S=\Kb[y^{1},y^{2},y^{3}]/\langle y^{1}y^{3}\rangle$.  The Graver basis of~$\ker_{\Zb}\Acal$ consists of
  $e_{\alpha_{1}}-e_{\alpha_{2}}$ for all~$\alpha_{1}\neq\alpha_{2}$.  According to
  Proposition~\ref{prop:Segre-generators}, the Segre product $R\SP S$ is isomorphic to
  \begin{equation*}
	\Kb[z^{11},z^{22},z^{33},z^{12},z^{13},z^{21},z^{23},z^{31},z^{32}]/K
  \end{equation*}
  for some ideal~$K$.  Here, $z^{\alpha\beta}$ is a simple generator if $\alpha=\beta$ and a Graver generator if~$\alpha\neq\beta$.
  Using properties of the tensor product, it is not difficult to see that
  \begin{equation*}
	K = \langle z^{1\beta}, z^{2\beta}z^{3\beta'}-z^{2\beta'}z^{3\beta}, z^{\alpha1}z^{\alpha'3} : \beta,\beta'\in\{1,2,3\},\alpha,\alpha'\in\{2,3\} \rangle.
  \end{equation*}
\end{ex}

\section{Relations between toric fiber products and Segre products}
\label{sec:relations}

There are many relations between toric fiber products and Segre products.  First,
Proposition~\ref{prop:Segre-generators} implies the following result from~\cite{TFP-II}:
\begin{lemma}
  There is a canonical injective homomorphism $\Kb[z]/(I\TFP J)\to R\SP S$.
  If $\codim(\Acal)=0$, then $\Kb[z]/(I\TFP J)\cong R\SP S$.
\end{lemma}
\begin{proof}
  $\Kb[z]/(I\TFP J)$ is isomorphic to the subalgebra of $R\SP S$ generated by the \hyperlink{simpGen}{simple
    generators}.  If $\codim(\Acal)=0$, then there are no Graver generators and $R\SP S$ is generated by the simple
  generators.
\end{proof}
Under assumptions on~$\Acal$ the toric fiber product can be computed from the Segre product:
\begin{lemma}
  \label{lem:canonical-TFP}
  Assume ~\eqref{eq:condA} and that all elements of $\Acal$ are pair-wise different.
  $\Kb[z]/(I\TFP J)$ is the subring of $R\SP S$ generated by $\sum_{a\in\Acal}R_{a}\otimes S_{a}$.  In other words,
  $\Kb[z]/(I\TFP J)$ is the subring of $R\SP S$ generated by all elements of total degree one.
\end{lemma}
\begin{proof}
  If~\eqref{eq:condA} holds, then $I,J$ are homogeneous with respect to the total grading, under which each variable
  $x^{\alpha}_{\beta},y^{\alpha}_{\gamma}$ has degree one.  The variables $z^{\alpha}_{\beta,\gamma}$ also have degree
  one, and hence $\Kb[z]/(I\TFP J)$ is generated in total degree one.  Moreover, any element of total degree one in
  $\Kb[x]\otimes\Kb[y]$ is a sum of monomials of the form $x^{\alpha}_{\beta}\otimes y^{\alpha}_{\gamma}$, and therefore
  each element of $R\SP S$ of total degree one lies in~$\Kb[z]/(I\TFP J)$.
\end{proof}

If~\eqref{eq:condA} is not imposed, then any Segre product is a quotient of a toric fiber product
(Proposition~\ref{prop:SP-as-TFP}).  However, not every toric fiber product is a Segre product.  As shown
in~\cite{TFP-II}, the Segre product of two integrally closed rings is integrally closed, but toric fiber products of
normal ideals are not necessarily normal, see Section~\ref{exs:norm-tfp-norm}.
\begin{prop}
  \label{prop:SP-as-TFP}
  For any Segre product $R\SP S$ there are a matrix $\Acal'$ with $\Nb\Acal'=\Nb\Acal$ and two $\Nb\Acal'$-homogeneous
  ideals $I'\subseteq\Kb[x']$ and $J'\subseteq\Kb[y']$ in $\Nb\Acal'$-graded rings $\Kb[x'],\Kb[y']$ such that
  \begin{align*}
    \Kb[x']/I' \cong R, & \qquad
    \Kb[y']/J' \cong S,  \\
    R\SP S &\cong \Kb[z']/(I'\TFP J').
  \end{align*}
\end{prop}
\begin{proof}
  Let $\Gcal$ be the Graver basis of~$\Acal$.  Let $\Kb[x,u]=\Kb[x][u^{g}_{\bbet}]$ and $\Kb[y,v]=\Kb[y][v^{g}_{\bgam}]$
  be the ring extensions of $\Kb[x]$ and $\Kb[y]$ with one additional variable for each $g\in\Gcal$ and each pair of
  Graver index sequences $\bbet$ and~$\bgam$.  Moreover, let $I'\subseteq\Kb[x,u]$ be the ideal generated by $I\Kb[x,u]$
  and the relations $u^{g}_{\bbet}-\prod_{i=1}^{n}x^{\alpha_{i}}_{\beta_{i}}$.  Similarly, let $J'\subseteq\Kb[y,v]$ be
  the ideal generated by $J\Kb[y,v]$ and the relations $v^{g}_{\bgam}-\prod_{j=1}^{m}y^{\alpha'_{j}}_{\gamma_{j}}$.
  Then $R\cong R':=\Kb[x,u]/I'$ and $S\cong S' := \Kb[y,v]/J'$.

  For each $g\in\Gcal$ let $a'_{g}=\sum_{\alpha:g_{\alpha}>0}c_{\alpha}a_{\alpha}$, and let
  $\Acal'=\Acal\cup\{a'_{g}:g\in\Gcal\}$.  Then $\Nb\Acal=\Nb\Acal'$, and the ideals $I'$ and $J'$ are homogeneous with
  respect to the $\Nb\Acal'$-grading that extends the $\Nb\Acal$-grading of $\Kb[x]$ and $\Kb[y]$ via
  $\deg(u^{g}_{\bbet}) = \deg(v^{g}_{\bgam}) = a'_{g}$.  The claim is that $R\SP S\cong \Kb[z']/(I'\TFP J')$.
  The natural injection $\Kb[x]\otimes\Kb[y]\to\Kb[x,u]\otimes\Kb[y,v]$ induces a map $R\otimes S\to R'\otimes S'$,
  which agrees with the isomorphism $R\otimes S\cong R'\otimes S'$ coming from the isomorphisms $R\cong R'$
  and~$S\cong S'$.  Hence, we obtain an injection $\iota:R\SP S\to R'\otimes S'$ which (upon restriction to the image)
  equals the isomorphism $\iota:R\SP S \xrightarrow{\sim} R'\SP[\Nb\Acal']S'$ of the two Segre products.  To complete
  the proof it suffices to show that $\iota$ maps each generator of $R\SP S$ to a generator of $\Kb[z']/(I'\TFP J')$.
  Observe that $\iota$ maps the simple generator $\ol M^{\alpha,\alpha}_{\beta,\gamma}\in R\SP S$ to the simple
  generator $\ol M^{\alpha,\alpha}_{\beta,\gamma}\in R'\SP S'$, and the Graver generator $\ol M^{g}_{\bbet,\bgam}\in
  R\SP S$ is mapped to
  \begin{equation*}
    \ol M^{g}_{\bbet,\bgam}
    = \prod_{i=1}^{n}\ol x^{\alpha_{i}}_{\beta_{i}}\otimes\prod_{j=1}^{m}\ol y^{\alpha'_{j}}_{\gamma_{j}}
    = \ol u^{g}_{\bbet}\otimes\ol v^{g}_{\bgam}\,,
  \end{equation*}
  which is a simple generator of~$R'\SP S'$.
\end{proof}
In concrete instances the proof of Proposition~\ref{prop:SP-as-TFP} can be simplified by choosing a smaller matrix
$\Acal'$ (Remark~\ref{rem:choice-Ap}). 
The crucial point is to ensure that the image of $\iota$ be contained in the set of simple generators.  It suffices to
check this for a minimal generating set. 

To sum up, the situation is the following: Let $\Lcal$ be an affine monoid, and let $R,S$ be $\Lcal$-graded affine
rings.  The Segre product $R\SP[\Lcal]S$ is well-defined.  Any choice of vector configuration $\Acal$ with
$\Nb\Acal=\Lcal$ and any choice of homogeneous generating sets $\{x^{\alpha}_{\beta}\}\subset R$ and
$\{y^{\alpha}_{\gamma}\}\subset S$ yields presentations $R\cong\Kb[x]/I$ and $S\cong\Kb[y]/J$ as quotients of polynomial
rings.  This data allows to construct a toric fiber product ring $\Kb[z]/(I\TFP J)$.  Each such toric fiber product is a
subring of the Segre product~$R\SP[\Lcal]S\cong(\Kb[x]/I)\SP(\Kb[y]/J)$.  By Proposition~\ref{prop:SP-as-TFP}, the Segre
product itself also arises in this way.  If $\Lcal$ has a generating set $\Acal$ satisfying~\eqref{eq:condA}, then any
$\Lcal$-graded ring has a total grading with values in~$\Nb$.  If $R$ and $S$ are generated in degree one, then, by
Lemma~\ref{lem:canonical-TFP}, a canonical choice of generators for the toric fiber product consists of the degree one
elements.

\begin{rem}
  According to~\cite[Theorem~3.1]{TFP-II}, if the ideals $I,J$ can be written as intersections $I=\bigcap_iI_i$ and
  $J=\bigcap_jJ_j$, then $I\times_\Acal J = \bigcap_i\bigcap_j I_i\times_\Acal J_j$.  If all ideals $I_i$ and $J_j$ are
  geometrically primary, then $I_i\times_\Acal J_j$ is also geometrically primary, and $I\times_\Acal J =
  \bigcap_i\bigcap_j I_i\times_\Acal J_j$ is a (possibly redundant) primary decomposition.

  A primary decomposition of the Segre product can be obtained similarly.  In the construction of
  Proposition~\ref{prop:SP-as-TFP}, if $I$ is (geometrically) primary, then $I'$ is also (geometrically)
  primary.  Therefore, decompositions $I=\bigcap_iI_i$ and $J=\bigcap_jJ_j$ into geometrically
  primary ideals yield decompositions $I'=\bigcap_iI'_i$ and~$J'=\bigcap_jJ'_j$ and thus a primary decomposition $I'\TFP
  J'=\bigcap_i\bigcap_j I'_i\TFP J'_j$.  Proposition~\ref{prop:SP-as-TFP} and its proof show that $I'\TFP J'$ is a
  defining ideal of~$R\SP S$ and that $I'_i\TFP J'_j$ is a defining ideal of~$(\Kb[x]/I_i)\SP(\Kb[y]/J_j)$.  In this
  sense, the Segre product preserves primary decompositions into geometrically primary ideals.

  The two decompositions $I\times_\Acal J = \bigcap_i\bigcap_j I_i\times_\Acal J_i$ and $I'\TFP J'=\bigcap_i\bigcap_j
  I'_i\TFP J'_j$ have the same number of components.  As discussed in~\cite{TFP-II}, these decompositions may contain
  redundant components, even if the original decompositions of $I$ and $J$ are not redundant.  Therefore, irredundant
  primary decompositions of the toric fiber product and the Segre product may have different numbers of components.  In
  fact, toric fiber products of different ideals with isomorphic quotient rings may have primary decompositions with
  different numbers of components:
\end{rem}
\begin{ex}
  Let $d=3$, $\Acal=(1,1,1)$ and $d_{\alpha}=1=d'_{\alpha}$ for~$\alpha=1,2,3$.  Consider the ideals
  \begin{equation*}
	I_{1}= \langle x^{1}\rangle, \qquad  I_{2}= \langle x^{2}\rangle, \qquad  J= \langle y^{1}y^3\rangle.
  \end{equation*}
  The graded rings $R_{1} :=\Kb[x^{1},x^{2},x^{3}]/I_{1}$ and $R_{2} :=\Kb[x^{1},x^{2},x^{3}]/I_{2}$ are isomorphic.
  The ideals $I_1$ and $I_2$ are prime, and the primary decomposition of $J$ is~$J=\langle y^1\rangle\cap\langle
  y^3\rangle$.  As in Example~\ref{ex:dep-on-pres}, the toric fiber product corresponds to the sum of the ideals, and so
  \begin{equation*}
	I_1\TFP J = \langle z^1 \rangle, \qquad
	I_2\TFP J = \langle z^2, z^1z^3 \rangle.
  \end{equation*}
  In particular, one toric fiber product is prime, while the other has a nontrivial primary decomposition.  By
  Example~\ref{ex:Segre-generators}, the two Segre products $R_1\SP S$ and $R_2\SP S$, where $S=\Kb[y]/J$, are
  isomorphic to
  \begin{equation*}
	\Kb[z^{22},z^{33},z^{21},z^{23},z^{31},z^{32}]
	/ \langle z^{2\beta}z^{3\beta'}-z^{2\beta'}z^{3\beta}, z^{\alpha1}z^{\alpha'3} : \beta,\beta'\in\{1,2,3\},\alpha,\alpha'\in\{2,3\} \rangle.
  \end{equation*}
  The defining ideal of the Segre product has two primary components.
\end{ex}
\begin{rem}\label{r:compatProj}
  If $I$ and $J$ are toric ideals, then Theorem~4.12 in~\cite{TFP-II} shows how to compute a generating set of $I\TFP J$
  from particular generating sets $\Fcal_1\subset I$ and~$\Fcal_2\subset J$ that satisfy a compatibility property,
  called \emph{compatible projection property}.  In principle, this theorem can be used together with
  Proposition~\ref{prop:SP-as-TFP} to compute a generating set of the defining ideal of the Segre product.  From the
  definitions of $I'$ and $J'$ it is direct how to construct generating sets $\Fcal'_1\subset I'$ and~$\Fcal'_2\subset
  J'$.  Unfortunately, $\Fcal'_1$ and $\Fcal'_2$ usually do not satisfy the compatible projection property, even if
  $\Fcal_1$ and $\Fcal_2$ have this property.
\end{rem}

\section{The neutral elements of toric fiber products and Segre products}
\label{sec:neutral-elements}

For fixed $\Acal$ both the toric fiber product and the Segre product are associative and commutative (up to natural
isomorphisms).  The next two lemmas show that the toric ideal \enlargethispage{3ex}
\begin{equation*}
  I_{\Acal} = \langle x^{c^{+}} - x^{c^{-}} : c^{+}-c^{-}\in\ker_{\Zb}\Acal\rangle \subseteq \Kb[x^{1},\dots,x^{d}]
\end{equation*}
is the neutral element of the toric fiber product, and the quotient $R_{\Acal}=\Kb[x^{1},\dots,x^{d}]/I_{\Acal}$ is the
neutral element of the Segre product.  Here, $x^{c^{+}}$ is shorthand for
$\prod_{\alpha=1}^{d}(x^{\alpha})^{c^{+}_{\alpha}}$.  Hence, the two constructions turn the set of all $\Nb\Acal$-graded
rings and the set of all $\Nb\Acal$-graded ideals into commutative monoids.

\begin{lemma}
  The ring $R_{\Acal}$ satisfies $R_{\Acal}\SP S\cong S$ for all $\Nb\Acal$-graded rings~$S$.
\end{lemma}
\begin{proof}
  By definition of~$I_{\Acal}$, the homogeneous component $(R_{\Acal})_{a}$ is one-dimensional for all~$a\in\Nb\Acal$.
  Therefore, $S_{a}\cong (R_{\Acal})_{a}\otimes S_{a}$ for all $a\in\Nb\Acal$, and thus~$S\cong R_{\Acal}\SP S$.
\end{proof}

\begin{lemma}
  Let $I_{\Acal}$ be the toric ideal defined by~$\Acal$.  Let $J \subset \Kb[y]$ be any $\Nb\Acal$-graded ideal.  Then
  $I_{\Acal}\times_{\Acal} J \cong J$.
\end{lemma}
\begin{proof}
  Under the assumptions of the lemma, the index $\beta$ can take only one value ($d'_{\alpha}=1$), so it can be omitted in
  the following.
  Let $\phi: \Kb[z] \to \Kb[x,y]/I_{\Acal}\cong R_{\Acal}\otimes\Kb[y], z^{\alpha}_{\gamma} \mapsto \ol{x^{\alpha}}\otimes y^{\alpha}_{\gamma}$ be the natural
  homomorphism and $\psi : \Kb[z] \to \Kb[y]$ the natural isomorphism.  We need to show that $\phi^{-1}(J(R_{\Acal}\otimes\Kb[y])) =
  \psi^{-1}(J)$.

  Let $f\in\psi^{-1}(J)$ be a homogeneous element of degree~$a\in\Nb\Acal$.  Then $\phi (f) = m_{a}\otimes f(y)$, where
  $m_{a}$ is the unique standard monomial of $R_{\Acal}$ of degree~$a$.  In particular, $\phi(f)\in J(R_{\Acal}\otimes\Kb[y])$, and thus
  $\psi^{-1}(J)\subseteq\phi^{-1}(J(R_{\Acal}\otimes\Kb[y]))$.
  Conversely, any element in the image of $\phi$ is of the form $g = \sum_{a\in\Nb\Acal} m_{a}\otimes f_{a}(y)$,
  where $f_{a}$ is a polynomial in~$\Kb[y]$ (and where all but a finite number of terms vanish in this sum).  Now, $g\in
  J(R_{\Acal}\otimes\Kb[y])$ if and only if $f_{a}\in J$ for all~$a\in\Nb\Acal$.  In this case $g=\phi(\psi^{-1}(\sum_{a}f))$,
  showing that $\phi^{-1}(J(R_{\Acal}\otimes\Kb[y]))\subseteq\psi^{-1}(J)$.  
\end{proof}

\section{Affine monoid rings and toric ideals}
\label{sec:toric-ideals}

Each vector configuration $\Bcal=(b^{\alpha}_{\beta})_{\alpha\in[d],\beta\in[d_{\alpha}]}\subset\Zb^{h_{1}}$ defines a
\emph{toric ideal}
\begin{equation*}
  I = I_{\Bcal} = \langle x^{c^{+}} - x^{c^{-}} : c^{+}-c^{-}\in\ker_{\Zb}\Bcal\rangle \subseteq \Kb[x]
\end{equation*}
which presents the affine monoid ring $\Kb[\Nb\Bcal] \cong \Kb[x]/I_{\Bcal}$.  Many algebraic properties of affine
monoid rings and toric ideals are combinatorial---they can be translated into combinatorial properties of the affine
monoid.  Corresponding properties of affine monoid rings and toric ideals are usually given the same name; and it is
then a theorem to see that an affine monoid $L$ has a property if and only if $\Kb[L]$ has the ring property with the
same name.  See~\cite[Theorem~4.42]{bruns09:poly_k_theory} for a very general version of this correspondence.
For instance, this works for the following definitions:
\begin{defi}\label{d:normalMonoid}
  If $L\subset K$ are affine monoids, then $L$ is \emph{integrally closed in $K$} if it is equal to its \emph{integral
	closure} $\{l \in K: nl \in L, n\in\Nb_{+} \}$ in~$K$.  An affine monoid $L$ is \emph{normal} if it is integrally
  closed in its universal group~$\gp(L)$.
\end{defi}

The statements of Sections~\ref{sec:generators-SP} and~\ref{sec:relations} have combinatorial formulations too.  To see
them, we first combinatorialize the product setup.  Let
$\Ccal=(c^{\alpha}_{\gamma})_{\alpha\in[d],\gamma\in[d_{\alpha}]}\subset\Zb^{h_{2}}$ be another integer vector
configuration with toric ideal~$I_{\Ccal}\subset\Kb[y]$.  Then $I=I_{\Bcal}$ and $J=I_{\Ccal}$ are
$\Nb\Acal$-homogeneous, if and only if there exist matrices $B,C$ such that
$a_{\alpha}=Bb^{\alpha}_{\beta}=Cc^{\alpha}_{\gamma}$ for all $\alpha,\beta,\gamma$.  As shown in~\cite{TFP-II}, the
toric fiber product of $I$ and $J$ is the toric ideal of the vector configuration
\begin{equation*}
  \Bcal\times_{\Acal}\Ccal
  = 
  \begin{pmatrix}
    \begin{pmatrix}
      b^{\alpha}_{\beta} \\ c^{\alpha}_{\gamma}
    \end{pmatrix}
    \in\Zb^{h_{1}+h_{2}}
  \end{pmatrix}_{\alpha\in[d],\beta\in[d_{\alpha}],\gamma\in[d'_{\alpha}]}.
\end{equation*}
For toric ideals the construction of Proposition~\ref{prop:SP-as-TFP} works as follows: Let $\Gcal$ be a Graver basis
of~$\Acal$.  For any $g\in\Gcal$ and any pair of Graver index sequences $\bbet,\bgam$ for $g$ let
\begin{equation*}
  a^g = \sum_i a^{\alpha_i} = \sum_j a^{\alpha'_j},
  \qquad
  b^g_\bbet = \sum_{i}b^{\alpha_i}_{\beta_i}
  \qquad\text{and}\qquad
  c^g_\bgam = \sum_{j}c^{\alpha'_j}_{\gamma_j}\,,
\end{equation*}
and consider the vector configurations $\Acal'=\Acal\cup(a^g)_{g\in\Gcal}$, $\Bcal'=
\Bcal\cup(b^g_\bbet)_{g\in\Gcal,\bbet}$ and $\Ccal'= \Ccal\cup(c^g_\bgam)_{g\in\Gcal,\bgam}$ (here, the symbol $\cup$
denotes the union of indexed multi-sets, i.e.~the same vector may appear multiple times in the same vector
configuration; cf.~Remark~\ref{rem:choice-Ap}).  Then $a^g = Bb^g_\bbet = Cc^g_\bgam$, and the toric ideals $I_{\Bcal'}$
and~$I_{\Ccal'}$ can be used as the ideals $I'$ and $J'$ in Proposition~\ref{prop:SP-as-TFP}.  Therefore, the Segre
product $\Kb[\Nb\Bcal] \SP \Kb[\Nb\Ccal]$
can be presented by the toric ideal of the vector configuration
$\Bcal'\TFP[\Acal'] \Ccal'$.
\begin{cor}
  \label{cor:torSPtoristor}
  The Segre product of affine monoid rings is an affine monoid ring.
\end{cor}
\begin{proof}
  If $I_{\Bcal}$ and $I_{\Ccal}$ are presenting toric ideals, then the toric fiber product of $I_{\Bcal'}$ and
  $I_{\Ccal'}$ above presents the Segre product by Proposition~\ref{prop:SP-as-TFP}.
\end{proof}

\begin{rem}
  \label{rem:choice-Ap}
  When applying Proposition~\ref{prop:SP-as-TFP}, there are usually several natural choices for $\Acal'$, $I'$ and~$J'$.
  For example, $\Acal'$ may contain duplicate vectors.  While it is possible to remove those, this will typically lead
  to a larger fiber product $\Bcal'\TFP[\Acal']\Ccal'$.  The reason is that $\Bcal'\TFP[\Acal']\Ccal'$ contains one
  vector for each pair of vectors from $\Bcal'$ and $\Ccal'$ having the same degree in~$\Acal'$.  If vectors in $\Acal'$
  are identified, then the number of pairs grows.  This effect is illustrated in Section~\ref{exs:hier-mod}.
\end{rem}
By construction the vector configuration $\Bcal\TFP\Ccal$ is a sub-configuration of $\Bcal'\TFP[\Acal']\Ccal'$.  The
elements of $\Bcal\TFP\Ccal$ are \emph{simple columns}.  For any pair of Graver index sequences~$\bbet,\bgam$ the vector
\begin{equation*}
  m^{g}_{\bbet,\bgam}
  = \begin{pmatrix}
	b^{g}_{\bbet} \\ c^{g}_{\bgam}
  \end{pmatrix}
\end{equation*}
is a \emph{Graver column}.  In other words, the Graver columns are the elements of
$\Bcal'\TFP[\Acal']\Ccal'\setminus\Bcal\TFP\Ccal$.  Some Graver columns correspond to redundant Graver generators.  The
following two lemmas give easy criteria for this.
\begin{lemma}
  Let $\Gcal$ be the Graver basis of $\Acal$, and let $\bbet,\bgam$ be a pair of Graver index sequences for
  some~$g\in\Gcal$.
  \begin{itemize}
  \item
	\label{lem:red-gen}
	$\ol M^{g}_{\bbet,\bgam}$ is a product of simple generators if and only if 
	$m^{g}_{\bbet,\bgam}$ is a non-negative integral combination of the simple columns.
  \item 
	\label{lem:Veronese-redundancy}
	If there exists a Graver index sequence $\bbet'$ for $-g$ such that $b^{g}_{\bbet}=b^{-g}_{\bbet'}$, then the Graver
	generator $\ol M^{g}_{\bbet,\bgam}$ is a product of simple generators.
  \end{itemize}
\end{lemma}
\begin{proof}
  The first statement translates the fact that an algebra generator of the affine monoid ring $R_{\Bcal}\SP R_{\Ccal}$ is
  redundant if and only if its exponent vector is redundant as a monoid generator.
  The assumptions of the second statement imply $c^{g}_{\bgam}=\sum_{j=1}^{m}c_{\gamma_{j}}^{\alpha_{j}'}$ and
  $b^{g}_{\bbet}=b^{-g}_{\bbet'}=\sum_{j=1}^{m}b_{\beta'_{j}}^{\alpha_{j}'}$.  Therefore,
  \begin{equation*}
	m^{g}_{\bbet,\bgam}
	= \begin{pmatrix}
	  b^{g}_{\bbet} \\ c^{g}_{\bgam}
	\end{pmatrix}
	= \sum_{j=1}^{m}
	\begin{pmatrix}
	  b^{\alpha_{j}'}_{\beta'_{j}} \\ c^{\alpha_{j}'}_{\gamma_{j}}
	\end{pmatrix},
  \end{equation*}
  and so the second statement follows from the first.
\end{proof}

By Serre's criterion and~\cite[Theorem~6]{Tousi2003672}, if $\Kb$ is perfect, then the tensor product of integrally
closed $\Kb$-algebras is integrally closed.  By \cite[Lemma~2.6]{TFP-II}, the Segre product inherits this property, since it is a direct summand in the tensor product.
In contrast, the toric fiber product does not preserve normality in general.  Assume that $R\otimes S$ is an integral
domain.  Then the integral closure of $\Kb[z]/(I\TFP J)$ is a subring of~$R\SP S$.  We may ask when $R\SP S$ equals the
integral closure of~$\Kb[z]/(I\TFP J)$.  Two things need to be checked:
\begin{enumerate}
\item Every Graver generator lies in the field of fractions of~$\Kb[z]/(I\TFP J)$.
\item Every Graver generator is integral over~$\Kb[z]/(I\TFP J)$.
\end{enumerate}
From a geometric point of view, the first statement says that $\Spec(R\SP S)$ is a dense subvariety of
\mbox{$\Spec(\Kb[z]/(I\TFP J))$}, while the second statement says that the natural inclusion $\Kb[z]/(I\TFP J)\to R\SP
S$ induces a finite map of affine schemes.  If one of these two statements holds, then the dimensions of
$\Spec(\Kb[z]/(I\TFP J))$ and $\Spec(R\SP S)$ agree.  In general, this need not be the case, see
Example~\ref{ex:dep-on-pres}.
The next two lemmas give criteria for these two statements to hold in the case of toric fiber products of toric ideals.
The proofs follow from standard arguments from the correspondence between affine monoid rings and affine monoids.
\begin{lemma}
  \label{lem:Graver-integral}
  Let $I_{\Bcal}$ and $I_{\Ccal}$ be toric ideals.  The Graver generator $\ol M^{g}_{\bbet,\bgam}$ is integral over
  $\Kb[z]/(I\TFP J)$ if and only if $m^{g}_{\bbet,\bgam}$ is a non-negative rational linear combination of the simple
  columns.  
\end{lemma}
\begin{proof}
  An integral relation of the Graver generator $\ol M$ can be taken to be of the form~$\ol M^{n}=\ol M'$, where $\ol M'$
  is a product of simple generators (this holds more generally for any monomial, see \cite[\S1.4 and
  Corollary~2.3.7]{SwaHuIntClosure}).  Such a relation is equivalent to the statement that $n m^{g}_{\bbet,\bgam}$ is a
  non-negative integral linear combination of the simple columns.
\end{proof}
\begin{lemma}
  \label{lem:Graver-quotient}
  Let $I_{\Bcal}$ and $I_{\Ccal}$ be toric ideals.  The Graver generator $\ol M^{g}_{\bbet,\bgam}$ lies in the field of
  fractions of $\Kb[z]/(I_\Bcal\TFP I_\Ccal)$ if and only if the Graver column $m^{g}_{\bbet,\bgam}$ can be written
  as an integral combination (not necessarily non-negative) of the simple columns.  
\end{lemma}
\begin{proof}
  $\ol M^{g}_{\bbet,\bgam}$ lies in the field of fractions of $R\SP S$ if and only if there exist $\ol M,\ol M'\in R\SP
  S$ with $\ol M^{g}_{\bbet,\bgam}\cdot \ol M'=\ol M$.  As in the proof of Lemma~\ref{lem:Graver-integral}, $\ol M$ and
  $\ol M'$ can be taken to be products of simple generators themselves.  The statement follows by passing to the
  exponent vectors.
\end{proof}

\section{Examples}
\label{sec:examples}

\subsection{Products of Veronese ideals over Veronese monoids}
\label{exs:Veronese-products}

Fix $k>0$ and~$n>0$.  For any multi-index $\biot=(\iota_{1},\dots,\iota_{k})$ of size $|\biot|=k$ with $\iota_{i}\in[n]$
let $v_{\biot}=\sum_{i=1}^{k}e_{\iota_{i}}\in\Zb^{n}$ be the sum of the unit
vectors~$e_{\iota_{1}},\dots,e_{\iota_{k}}$.  Let
\begin{equation*}
\Vcal_{k,n}=\{v_{\biot} : |\biot|=k, \iota_{i}\in[n]\}
\end{equation*}
be the configuration of all non-negative integer vectors $v\in\Nb_0^n$ with $\sum_{i=1}^nv_i=k$.  The toric variety of
the toric ideal $I_{k,n}:=I_{\Vcal_{k,n}}$ is the \emph{Veronese variety} of degree $k$ in $n$ variables.

Fix $k>0$, and let $n_1>1$, and consider a partition $[n_1] = N_1\dot\cup\dots\dot\cup N_{n_0}$ into $n_0>1$ (disjoint)
blocks.  Such a partition defines a natural map $p_1:[n_1]\to[n_0]$ such that $x\in N_{p_1(x)}$ for all~$x\in[n_1]$.
This map extends to a linear map~$\Zb^{n_1}\to\Zb^{n_0}$ by sending the $i$th unit vector to the $p_1(i)$th unit vector,
and this map is denoted by the same symbol.  Note that $p_1$ sends each vector in $\Vcal_{k,n_1}$ to a vector
in~$\Vcal_{k,n_0}$.  Therefore, $p_1$ induces an $\Nb\Vcal_{k,n_0}$-grading under which the toric ideal $I_{k,n_1}$ is
homogeneous.

\begin{prop}
  \label{prop:product-with-Veronese}
  Consider an $\Nb\Vcal_{k,n_{0}}$-graded polynomial ring $\Kb[y]$ where the degrees of the variables belong
  to~$\Vcal_{k,n_{0}}$, and let $J$ be a homogeneous ideal.  Then the toric fiber product
  $I_{k,n_1}\TFP[\Vcal_{k,n_0}]J$ is a defining ideal of the corresponding Segre product.
\end{prop}
\begin{proof}
  We want to apply Lemma~\ref{lem:Veronese-redundancy} as follows: Any Graver basis element~$g$ (in
  fact, any kernel element) corresponds to a non-negative vector $v_{g}\in\Zb^{n_{0}}$ and two different ways of
  writing $v_{g}$ as a sum of $k$-tuples of unit-vectors.  The vector $b^{g}_{\bbet}$ arises from $v_{g}$ by replacing
  each unit vector $e_{i}$, with $i\in[n_{0}]$, by some $e_{j}$ with $j\in[n_{1}]$ satisfying $p_{1}(e_{j})=e_{1}$.  The
  crucial point is that the $k$-tuples can be re-arranged such that $b^{g}_{\bbet}$ can be written in the
  form~$b^{-g}_{\bbet'}$.
  The technical details are as follows:

  Let $\Gcal$ be the Graver basis of $\Vcal_{k,n_{0}}$, and let~$g\in\Gcal$.  By definition, $g\in\ker\Vcal_{k,n_{0}}$,
  and so $v_{g}:=\sum_{\biot:g_{\biot}>0}g_{\biot}v_{\biot} = - \sum_{\biot:g_{\biot}<0}g_{\biot}v_{\biot}$.  Let
  $\bbet,\bgam$ be a pair of Graver index sequences for~$g$.  The vector $b^{g}_{\bbet}$ satisfies
  $p_{1}(b^{g}_{\bbet})=v_{g}$.
  Denote by $m$ the sum of the positive entries of~$g$.  Since all vectors $v_{\iota}\in\Vcal_{k,n_{0}}$ satisfy
  $\sum_{i=1}^{n_{0}}(v_{\iota})_{i}=k$, the sum of the negative entries of~$g$ equals~$-m$.

  Write $v_{g}=\sum_{j=1}^{m}v_{\alpha_{j}'}$, where each $\alpha_{j}'=(\iota_{j,1},\dots,\iota_{j,k})$ is a multi-index
  of size~$k$ and where $\bal'=(\alpha_{1}',\dots,\alpha_{m}')$ is the multi-index associated with~$g$, as in
  Section~\ref{sec:generators-SP}.  Note that $\bbet=(\beta_{1},\dots,\beta_{m})$ also consists of multi-indices
  $\beta_{j}=(\lambda_{j,1},\dots,\lambda_{j,k})$ in such a way that $\lambda_{j,i}\in[n_{1}]$
  and~$p_{1}(\lambda_{j,i})=\iota_{j,i}$.

  Define a map $\kappa:[m]\times[k]\to[n_{1}]$ iteratively as follows:
  \begin{align*}
	\kappa(j,i) &= \min\big\{
	\{ \lambda_{j',i'} : p_{1}(\lambda_{j',i'})=\iota_{j,i}\} \setminus \{ \kappa_{j',i'} : kj'+i'< kj+i\}
	\big\} \\
	&= \min\big\{
	\{ \lambda_{j',i'} : p_{1}(\lambda_{j',i'})=\iota_{j,i}\} \setminus \{ \kappa_{j',i'} : kj'+i'< kj+i, p_{1}(\kappa_{j',i'})=\iota_{j,i}\}
	\big\},
  \end{align*}
  where the symbol $\setminus$ denotes the difference of multi-sets.  The second line shows that $\kappa$ is well-defined:
  By the recursive definition, the cardinality of $\{ \kappa_{j',i'} : kj'+i'< kj+i, p_{1}(\kappa_{j',i'})=\iota_{j,1}\}$
  is strictly less than the cardinality of $\{ \lambda_{j',i'} : p_{1}(\lambda_{j',i'})=\iota_{j,i}\}$, and hence the
  difference of multi-sets inside the minimum can never be empty.

  Let $\beta'_{j}=(\kappa(j,1),\dots,\kappa(j,k))$ and~$\bbet'=(\beta'_{1},\dots,\beta'_{m})$.  The claim is that
  $\bbet'$ is a \hyperlink{GIS}{Graver index sequence} for $-g$, and that $b^{g}_{\bbet}=b^{-g}_{\bbet'}$.  The first
  claim follows since $p_{1}(\kappa_{j,i})=\iota_{j,i}$ and~$p_{1}(v_{\beta'_{j}})=v_{\alpha'_{j}}$.  To prove the second
  claim, observe that the two multi-sets $\{\kappa(j,i):j=1,\dots,m,i=1,\dots,k\}$ and
  $\{\lambda_{j,i}:j=1,\dots,m,i=1,\dots,k\}$ are identical, and hence
  \begin{equation*}
	b^{g}_{\bbet}=\sum_{j=1}^{m}v_{\beta_{j}}
	=\sum_{j=1}^{m}\sum_{i=1}^{k}e_{\lambda_{j,i}}
	=\sum_{j=1}^{m}\sum_{i=1}^{k}e_{\kappa(j,i)}
	=\sum_{j=1}^{m}v_{\beta'_{j}}=b^{-g}_{\bbet'}.
  \end{equation*}
  The statement of the proposition now follows from Lemma~\ref{lem:Veronese-redundancy}.
\end{proof}
\begin{cor}
  Iterated toric fiber products of Veronese ideals over Veronese monoids (using partitions to define the common
  gradings) are normal.
\end{cor}
\begin{proof}
  The statement follows by induction on the number of factors.  By Proposition~\ref{prop:product-with-Veronese}, the
  toric fiber product with a Veronese ideals is the defining ideal of a Segre product, and the Segre product preserves
  normality.
\end{proof}

\subsection{The toric fiber product does not preserve normality}
\label{exs:norm-tfp-norm}

The question of normality of toric fiber products arises in algebraic statistics.  Knowing that the semigroup of
possible margins of a hierarchical model for contingency table is normal can significantly simplify the statistical
analysis~\cite{yuguo06:_sequen,takemura08:_holes}.  In~\cite{sullivant09:_normal_graph} Sullivant shows that a binary
graph model---a certain hierarchical model defined for an undirected graph $G$---is normal if and only if $G$ is free of
$K_{4}$ minors.  Since gluing graphs corresponds to taking toric fiber products~\cite[Section~5]{TFP-II}, we can produce
non-normal toric fiber products of normal ideals.  These ideals are usually in many variables.  Here we give a much
smaller example.

Let
\begin{equation*}
  \Acal =
  \begin{pmatrix}
    2 & 0 & 1 \\
    0 & 2 & 1
  \end{pmatrix}.
\end{equation*}
Then $a_{1}+a_{2}=2 a_{3}$.  Therefore, $\codim(\Acal)=1$, and $\ker_{\Zb}\Acal$ is generated by~$h=(1,1,-2)$.

Consider the rings $\Kb[x]=\Kb[x^{1}_{1},x^{1}_{2},x^{2},x^{3}]$
and~$\Kb[y]=\Kb[y^{1}_{1},y^{1}_{2},y^{2}_{1},y^{2}_{2},y^{3}]$ and the ideals
$I=\langle{}x^{1}_{2}x^{2}-(x^{3})^{2}\rangle$ and~$J=\langle{}(y^{3})^{4}-y^{1}_{1}y^{1}_{2}y^{2}_{1}y^{2}_{2}\rangle$.
In fact, $I=I_{\Bcal}$ and $J=I_{\Ccal}$ are the toric ideals of the two matrices
\begin{align*}
  \Bcal &=
  \begin{pmatrix}
    0 & 2 & 0 & 1 \\
    0 & 0 & 2 & 1 \\
    1 & 0 & 0 & 0
  \end{pmatrix}, &
  \Ccal &=
  \begin{pmatrix}
    4 & 0 & 0 & 0 & 1 \\
    0 & 4 & 0 & 0 & 1 \\
    0 & 0 & 4 & 0 & 1 \\
    0 & 0 & 0 & 4 & 1
  \end{pmatrix}.
\end{align*}
As can be seen from the defining equations, both $I$ and $J$ are $\Nb\Acal$-homogeneous with respect to the grading
determined by the superscripted indices.  This also follows from the equalities
\begin{equation*}
  \Acal =
  \begin{pmatrix}
    1 & 0 & 2 \\
    0 & 1 & 0
  \end{pmatrix}
  \Bcal = \frac12
  \begin{pmatrix}
    1 & 1 & 0 & 0 \\
    0 & 0 & 1 & 1
  \end{pmatrix}
  \Ccal
\end{equation*}
Moreover, the two rings $R = \Kb[x]/I$ and $S = \Kb[y]/J$ are normal.

There are six Graver generators of $R\SP S$:
\begin{align*}
  \ol M_{\beta} &= \ol x^{1}_{\beta}\ol x^{2}\otimes(\ol y^{3})^{2}, &
  \ol M_{\gamma,\gamma'} &= (\ol x^{3})^{2}\otimes\ol y^{1}_{\gamma}\ol y^{2}_{\gamma'}.
\end{align*}
The generator $\ol M_{1}$ lies in the Segre product~$R\SP S$, but not in $\Kb[z]/(I\TFP J)$.  This can be seen, for
example, from Lemma~\ref{lem:red-gen}.
All other Graver generators are, in fact, redundant.  The square
\begin{equation*}
  (\ol M_{1})^{2}= (\ol x^{1}_{1})^{2}(\ol x^{2})^{2}\otimes(\ol y^{3})^{4} = (\ol x^{1}_{1})^{2}(\ol x^{2})^{2}\otimes\ol y^{1}_{1}\ol y^{1}_{2}\ol y^{2}_{1}\ol y^{2}_{2}
  = (\ol x^{1}_{1}\otimes\ol y^{1}_{1})(\ol x^{1}_{1}\otimes\ol y^{1}_{2})(\ol x^{2}\otimes\ol y^{2}_{1})(\ol x^{2}\otimes\ol y^{2}_{2})
\end{equation*}
lies in~$\Kb[z]/(I\TFP J)$.  Therefore, $\ol M_{1}$ is integral over~$\Kb[z]/(I\TFP J)$.  Moreover,
\begin{equation*}
  \ol M_{1}\cdot (\ol x^{1}_{2}\otimes\ol y^{1}_{1})
  = \ol x^{1}_{1}\ol x^{1}_{2} \ol x^{2}\otimes\ol y^{1}_{1}(\ol y^{3})^{2}
  = \ol x^{1}_{1}(\ol x^{3})^{2}\otimes\ol y^{1}_{1}(\ol y^{3})^{2}
  = (\ol x^{1}_{1}\otimes\ol y^{1}_{1})(\ol x^{3}\otimes\ol y^{3})^{2}
\end{equation*}
lies in~$\Kb[z]/(I\TFP J)$, and so $\ol M_{1}$ belongs to the total ring of fractions of~\mbox{$\Kb[z]/(I\TFP J)$}.
Hence the toric fiber product~$\Kb[z]/(I\TFP J)$ is not normal.  By Lemmas~\ref{lem:Graver-integral}
and~\ref{lem:Graver-quotient}, the normalization of \mbox{$\Kb[z]/(I\TFP J)$} equals the Segre product~$R\SP S$.

\subsection{A hierarchical model}
\label{exs:hier-mod}

In this section we compute the products of two hierarchical models to illustrate the construction of
Proposition~\ref{prop:SP-as-TFP}.  In the notation of~\cite{TFP-II}, the two ideals $I$ and $J$ are the toric ideals
describing the hierarchical models of the simplicial complexes $\Gamma_I=[13][23]$ and $\Gamma_J=[14][24]$.  The toric
fiber product is then the toric ideal of the hierarchical model of the simplicial complex
$\Gamma_{I\times_{\Acal}J}=[13][23][14][24]$.  We do not explain this notation here, since it is not needed to
understand the example.

Let
\begin{equation*}
  \Acal=
  \begin{pmatrix}
    1 & 1 & 1 & 1 \\
    1 & 1 & 0 & 0 \\
    1 & 0 & 1 & 0
  \end{pmatrix},
\end{equation*}
and assume $d_{\alpha}=d'_{\alpha}=2$ for all $\alpha,\alpha'\in[d]=\{1,2,3,4\}$.  The codimension of $\Acal$ is one,
and $\ker_{\Zb}\Acal$ is generated (as a lattice) by the single vector
$ 
  h = (1, -1, -1, 1).
$ 
A minimal Graver basis is~$\Gcal=\{\pm h\}$.

The two ideals
\begin{align*}
  I &= \langle x^{1}_{\beta}x^{4}_{\beta} - x^{2}_{\beta}x^{3}_{\beta} : \beta=1,2 \rangle \\
  J &= \langle y^{1}_{\gamma}y^{4}_{\gamma} - y^{2}_{\gamma}y^{3}_{\gamma} : \gamma=1,2 \rangle
\end{align*}
are $\Nb\Acal$-homogeneous.  In fact, both $I$ and $J$ are isomorphic to the toric ideal of the matrix
\begin{equation*}
  \Bcal =
  \begin{pmatrix}
    1 & 1 & 1 & 1 & 1 & 1 & 1 & 1 \\
    1 & 0 & 1 & 0 & 0 & 0 & 0 & 0 \\
    0 & 1 & 0 & 1 & 0 & 0 & 0 & 0 \\
    0 & 0 & 0 & 0 & 1 & 0 & 1 & 0 \\
    1 & 0 & 0 & 0 & 1 & 0 & 0 & 0 \\
    0 & 1 & 0 & 0 & 0 & 1 & 0 & 0
  \end{pmatrix}
\end{equation*}
The toric fiber product ring $\Kb[z]/(I\TFP J)$ is generated by the sixteen simple generators
\begin{equation*}
  \oxab\otimes\oyag,\qquad\text{ for all }\alpha,\beta,\gamma.
\end{equation*}
There are 32 Graver monomials
\begin{equation*}
  x^{1}_{\beta_{1}}x^{4}_{\beta_{2}}\otimes y^{2}_{\gamma_{1}}y^{3}_{\gamma_{2}}
  \quad\text{ and }\quad
  x^{2}_{\beta_{1}}x^{3}_{\beta_{2}}\otimes y^{1}_{\gamma_{1}}y^{4}_{\gamma_{2}}.
\end{equation*}
However, the relations in the ideals $I$ and $J$ imply that the corresponding Graver generators are not independent:
Namely, all Graver generators with $\beta_{1}=\beta_{2}$ or $\gamma_{1}=\gamma_{2}$ are, in fact, products of the simple
generators.  In total, the Segre product $R\SP S$ is generated by the sixteen simple generators and the eight Graver
generators with $\beta_{1}\neq\beta_{2}$ and $\gamma_{1}\neq\gamma_{2}$.  Using 
Lemmas~\ref{lem:Graver-integral} and~\ref{lem:Graver-quotient} one can show that these eight Graver generators are not
integral over~$\Kb[z]/(I\TFP J)$; but they lie in the field of fractions of~$\Kb[z]/(I\TFP J)$: In fact,
\begin{align*}
  (\ol x^{1}_{\beta_{1}}\ol x^{4}_{\beta_{2}}\otimes\ol y^{2}_{\gamma_{1}}\ol y^{3}_{\gamma_{2}})
  \cdot (\ol x^{1}_{\beta_{2}}\otimes\ol y^{1}_{1})
  &= \ol x^{1}_{\beta_{1}}\ol x^{2}_{\beta_{2}}\ol x^{3}_{\beta_{2}}\otimes\ol y^{1}_{1}\ol y^{2}_{\gamma_{1}}\ol y^{3}_{\gamma_{2}}, \\
  (\ol x^{2}_{\beta_{1}}\ol x^{3}_{\beta_{2}}\otimes\ol y^{1}_{\gamma_{1}}\ol y^{4}_{\gamma_{2}})
  \cdot (\ol x^{2}_{\beta_{2}}\otimes\ol y^{2}_{1})
  &= \ol x^{2}_{\beta_{1}}\ol x^{1}_{\beta_{2}}\ol x^{4}_{\beta_{2}}\otimes\ol y^{2}_{1}\ol y^{1}_{\gamma_{1}}\ol y^{4}_{\gamma_{2}}.
\end{align*}

The construction of Proposition~\ref{prop:SP-as-TFP} yields
\begin{equation*}
  \Acal'=
  \left(
  \begin{matrix}
    1 & 1 & 1 & 1 \\
    1 & 1 & 0 & 0 \\
    1 & 0 & 1 & 0
  \end{matrix}
  \;\middle|\;
  \begin{matrix}
    2 & 2 \\
    1 & 1 \\
    1 & 1
  \end{matrix}
  \right),
\end{equation*}
where the last two columns correspond to~$\pm h$.  Moreover,
\begin{subequations}
  \begin{align}
    \Kb[x,u]&=\Kb[x]         [u^{h}_{\beta_{1},\beta_{2}},u^{-h}_{\beta_{1},\beta_{2}}: \beta_{1},\beta_{2}=1,2],
    \label{eq:Kbxu}    \\
    \Kb[y,v]&=\Kb[\wwo{y}{x}][v^{h}_{\gama_{1},\gama_{2}},v^{-h}_{\gama_{1},\gama_{2}}: \gama_{1},\gama_{2}=1,2],
    \label{eq:Kbyv}\\
    I' &= I\Kb[x,u] + \langle u^{h}_{\beta_{1},\beta_{2}}-x^{1}_{\beta_{1}}x^{4}_{\beta_{2}}, u^{-h}_{\beta_{1},\beta_{2}}-x^{2}_{\beta_{1}}x^{3}_{\beta_{2}} : \beta_{1},\beta_{2}=1,2 \rangle, \\
    J' &= I\Kb[\wwo{y}{x},v] + \langle v^{h}_{\gama_{1},\gama_{2}}-\wwo{y}{x}^{2}_{\gama_{1}}\wwo{y}{x}^{3}_{\gama_{2}},
    v^{-h}_{\gama_{1},\gama_{2}}-\wwo{y}{x}^{1}_{\gama_{1}}\wwo{y}{x}^{4}_{\gama_{2}} : \gamma_{1},\gamma_{2}=1,2\rangle.
    \label{eq:Jp}    
  \end{align}
\end{subequations}
In fact, using the defining relations of $I$ and $J$ it is possible to simplify the construction by removing the
generators $u^{\pm h}_{\beta_{1},\beta_{2}}$ and $v^{\pm h}_{\gamma_{1},\gamma_{2}}$ with $\beta_{1}=\beta_{2}$ and
$\gamma_{1}=\gamma_{2}$.  Therefore, we may additionally require $\beta_{1}\neq\beta_{2}$ in~\eqref{eq:Kbxu}
to~\eqref{eq:Jp}.
Hence $I'$ and $J'$ are isomorphic to the toric ideal of the matrix
\begin{equation*}
  \Bcal' =
  \left(\begin{matrix}
    1 & 1 & 1 & 1 & 1 & 1 & 1 & 1 \\
    1 & 0 & 1 & 0 & 0 & 0 & 0 & 0 \\
    0 & 1 & 0 & 1 & 0 & 0 & 0 & 0 \\
    0 & 0 & 0 & 0 & 1 & 0 & 1 & 0 \\
    1 & 0 & 0 & 0 & 1 & 0 & 0 & 0 \\
    0 & 1 & 0 & 0 & 0 & 1 & 0 & 0
  \end{matrix}
  \;\middle|\;
  \begin{matrix}
    2 & 2 \\
    1 & 0 \\
    0 & 1 \\
    0 & 1 \\
    1 & 0 \\
    0 & 1
  \end{matrix}
  \;\middle|\;
  \begin{matrix}
    2 & 2 \\
    1 & 0 \\
    0 & 1 \\
    0 & 1 \\
    0 & 1 \\
    1 & 0
  \end{matrix}
  \right).
\end{equation*}
The two additional blocks correspond to~$\pm h$.

For each $\alpha=1,\dots,6$ there are two variables in $\Kb[x,u]$ resp.~$\Kb[y,v]$ of degree~$a_{\alpha}$, where
$a_{\alpha}$ denotes the $\alpha$th column of~$\Acal'$.  Therefore, the toric fiber product of $I'$ and $J'$ is a toric
ideal in $2\cdot2\cdot6=24$ variables with matrix
\begin{equation*}
  \Bcal'\times_{\Acal'}\Bcal'
  =
  \left(
    \everymath{\scriptstyle}
  \begin{matrix}
    1 & 1 & 1 & 1 & 1 & 1 & 1 & 1  \\
    1 & 1 & 0 & 0 & 1 & 1 & 0 & 0  \\
    0 & 0 & 1 & 1 & 0 & 0 & 1 & 1  \\
    0 & 0 & 0 & 0 & 0 & 0 & 0 & 0  \\
    1 & 1 & 0 & 0 & 0 & 0 & 0 & 0  \\
    0 & 0 & 1 & 1 & 0 & 0 & 0 & 0  \\
    1 & 1 & 1 & 1 & 1 & 1 & 1 & 1  \\
    1 & 0 & 1 & 0 & 1 & 0 & 1 & 0  \\
    0 & 1 & 0 & 1 & 0 & 1 & 0 & 1  \\
    0 & 0 & 0 & 0 & 0 & 0 & 0 & 0  \\
    1 & 0 & 1 & 0 & 0 & 0 & 0 & 0  \\
    0 & 1 & 0 & 1 & 0 & 0 & 0 & 0 
  \end{matrix}
  \;\middle.\;
    \everymath{\scriptstyle}
  \begin{matrix}
    1 & 1 & 1 & 1 & 1 & 1 & 1 & 1 \\
    0 & 0 & 0 & 0 & 0 & 0 & 0 & 0 \\
    0 & 0 & 0 & 0 & 0 & 0 & 0 & 0 \\
    1 & 1 & 0 & 0 & 1 & 1 & 0 & 0 \\
    1 & 1 & 0 & 0 & 0 & 0 & 0 & 0 \\
    0 & 0 & 1 & 1 & 0 & 0 & 0 & 0 \\
    1 & 1 & 1 & 1 & 1 & 1 & 1 & 1 \\
    0 & 0 & 0 & 0 & 0 & 0 & 0 & 0 \\
    0 & 0 & 0 & 0 & 0 & 0 & 0 & 0 \\
    1 & 0 & 1 & 0 & 1 & 0 & 1 & 0 \\
    1 & 0 & 1 & 0 & 0 & 0 & 0 & 0 \\
    0 & 1 & 0 & 1 & 0 & 0 & 0 & 0 
  \end{matrix}
  \;\middle.\;
    \everymath{\scriptstyle}
  \begin{matrix}
    2 & 2 & 2 & 2 & 2 & 2 & 2 & 2 \\
    1 & 1 & 0 & 0 & 1 & 1 & 0 & 0 \\
    0 & 0 & 1 & 1 & 0 & 0 & 1 & 1 \\
    0 & 0 & 1 & 1 & 0 & 0 & 1 & 1 \\
    1 & 1 & 0 & 0 & 0 & 0 & 1 & 1 \\
    0 & 0 & 1 & 1 & 1 & 1 & 0 & 0 \\
    2 & 2 & 2 & 2 & 2 & 2 & 2 & 2 \\
    1 & 0 & 1 & 0 & 1 & 0 & 1 & 0 \\
    0 & 1 & 0 & 1 & 0 & 1 & 0 & 1 \\
    0 & 1 & 0 & 1 & 0 & 1 & 0 & 1 \\
    0 & 1 & 0 & 1 & 1 & 0 & 1 & 0 \\
    1 & 0 & 1 & 0 & 0 & 1 & 0 & 1
  \end{matrix}
  \right),
\end{equation*}
while the toric fiber product of $I$ and $J$ is a toric ideal in $2\cdot2\cdot4=16$ variables.

In this construction it is possible to choose a smaller matrix $\Acal'$ by identifying the last two columns of~$\Acal'$.
However, this leads to a larger polynomial ring $\Kb[z'']$ and thus enlarges the toric fiber product ideal, since in
this case there are more pairs $(u^{\pm h}_{\beta_{1},\beta_{2}},v^{\pm h}_{\gamma_{1},\gamma_{2}})$ of variables of
$\Kb[x,u]$ and $\Kb[y,v]$ that have the same degree.  In fact, in this example the toric fiber product of $I'$ and $J'$
over the matrix
\begin{equation*}
  \Acal''=
  \left(
  \begin{matrix}
    1 & 1 & 1 & 1 \\
    1 & 1 & 0 & 0 \\
    1 & 0 & 1 & 0
  \end{matrix}
  \;\middle|\;
  \begin{matrix}
    2 \\
    1 \\
    1
  \end{matrix}
  \right),
\end{equation*}
is a toric ideal in $2\cdot2\cdot4 + 4\cdot 4=32$ variables.  Of course,
\begin{equation*}
  \Kb[z'']/(I'\TFP[\Acal''] J')
  \cong
  \Kb[z']/(I'\TFP[\Acal'] J')
  \cong
  (\Kb[x]/I) \SP[\Nb\Acal'] (\Kb[y]/J).
\end{equation*}

\bibliographystyle{IEEEtranS}
\bibliography{math}

\end{document}